\documentclass[reqno,12pt,twoside]{article}
\usepackage{fullpage}
\usepackage{amssymb}
\usepackage{amsmath}
\usepackage{amsthm}
\numberwithin{equation}{section}
\newtheorem{theorem}{Theorem}[section]
\newtheorem{lemma}[theorem]{Lemma}

\newcommand{\Nzero}{\mathbb N_0}

\begin{document}

\title{\Large On a problem of Nathanson related to minimal asymptotic bases and maximal asymptotic nonbases
}\author{\large Shi-Qiang Chen\thanks{This work is supported by the National Natural Science Foundation of China (Grant No.
12301003), the Anhui Provincial Natural Science Foundation (Grant No. 2308085QA02) and
the University Natural Science Research Project of Anhui Province (Grant No. 2022AH050171).}}
\date{} \maketitle
 \vskip -3cm
\begin{center}
\vskip -1cm { \small
\begin{center}
 School of Mathematics and Statistics, Anhui Normal
University
\end{center}
\begin{center}
Wuhu 241002, PR China
\end{center}}
\end{center}

{\bf Abstract.} Let $\Nzero=\{0,1,2,\ldots\}$ and let $h\ge2$ be an integer. For a set $A\subseteq\Nzero$, write $hA$ for the set of all sums of $h$, not necessarily distinct, elements of $A$. In this paper, we prove that for every $h\ge2$, there is a partition $\Nzero=A\sqcup B$ such that $A$ is a minimal asymptotic basis of order $h$ and $B$ is a maximal asymptotic nonbasis of order $h$. This solves an open problem posed by Nathanson in 1974.



\section{Introduction}

Let $\Nzero=\{0,1,2,\ldots\}$. For an integer $h\ge2$ and a set $A\subseteq\Nzero$, define
\[
 hA=\{a_1+\cdots+a_h:a_1,\ldots,a_h\in A\}.
\]
We call $A$ an \emph{asymptotic basis of order $h$} if every sufficiently large integer belongs to $hA$; otherwise, $A$ is called an \emph{asymptotic nonbasis of order $h$}. An asymptotic basis $A$ is \emph{minimal} if $A\setminus\{a\}$ is not an asymptotic basis of order $h$ for every $a\in A$. Correspondingly, an asymptotic nonbasis $B$ is \emph{maximal} if $B\cup\{b\}$ is an asymptotic basis of order $h$ for every nonnegative integer $b\notin B$.

The well-known Erd\H{o}s--Tur\'an conjecture \cite{ErdosTuran1941} is closely related to the study of asymptotic bases of order $2$. In 1955, St\"ohr \cite{Stohr1955} introduced the concept of minimal asymptotic bases. In 1956, H\"artter \cite{Hartter1956} proved that minimal asymptotic bases of order $h$ exist for every integer $h\ge2$.

In 1974, Nathanson \cite{Nathanson1974} gave several examples of minimal asymptotic bases and also constructed an asymptotic basis of order $2$ containing no subset which is a minimal asymptotic basis of order $2$. Nathanson \cite{Nathanson1974} also posed the following famous problem.\\
{\bf Nathanson's Problem.}
Can $\Nzero$ be partitioned into two sets, one of which is a minimal asymptotic basis of order $h$ and the other a maximal asymptotic nonbasis of order $h$?

In 1976, Erd\H{o}s and Nathanson \cite{ErdosNathanson1976} proved that the nonnegative integers can be partitioned into a minimal asymptotic basis of order $2$ and a maximal asymptotic nonbasis of order $2$. For other related results on minimal asymptotic bases and maximal asymptotic nonbases, see [1-6, 8, 10-13, 15, 16, 18].

In this paper, we solve {\bf Nathanson's Problem.}

\begin{theorem}\label{thm1}
For every integer $h\ge2$, there is a partition
\[
 \Nzero=A\sqcup B
\]
such that $A$ is a minimal asymptotic basis of order $h$ and $B$ is a maximal asymptotic nonbasis of order $h$.
\end{theorem}

\section{Lemmas}

For the remainder of this section, assume that $h\ge3$, and put $d=10h$ and $T_0=100h^2$.
For integer intervals we use the convention $[u,v]=\{n\in\mathbb Z:u\le n\le v\}$. 

\begin{lemma}\label{lem1}
Let $S\subseteq\Nzero$, let $p,u,v$ be nonnegative integers with $p<u\le v$. If $p\in S$, $[u,v]\subseteq S$ and $v\ge 2u-p-1$,
then for every integer $m\ge2$, we have
\[
 [u+(m-1)p,mv]\subseteq mS.
\]
In particular, if $0\in S$ and $v\ge2u-1$, then
\[
 [u,mv]\subseteq mS.
\]
\end{lemma}

\begin{proof}
For $1\le r\le m$, write
\[
 I_r=r[u,v]+(m-r)p
   =[ru+(m-r)p,\,rv+(m-r)p]\subseteq mS.
\]
Noting that $v\ge 2u-p-1$, we have
\[
 (r+1)u+(m-r-1)p\le rv+(m-r)p+1.
\]
It follows that
\[
 [u+(m-1)p,mv]=\bigcup_{1\le r\le m} I_r\subseteq mS.
\]
This completes the proof of Lemma \ref{lem1}.
\end{proof}

\begin{lemma}\label{lem2}
Let $T$ and $U$ be two positive integers with $T\ge T_0$ and $T\ge2U-1$, and let $C$ and $D$ be two sets such that
\[
 0\in D,\qquad 1,2\in C,\qquad C\sqcup D=[0,T],\qquad [U,T]\subseteq D,\qquad
 [T+1,T+h-1]\subseteq hC.
\]
Put
\[
 L=2T,\qquad M=(2h-3)T+1,\qquad Q=2hM+1.
\]
For each $c\in C$, put
\[
 x_c=Q-(h-2)L-c.
\]
Put
\[
 G=\{L\}\cup[M,2M]\cup\bigcup_{c\in C}\{x_c\}.
\]
We define
\[
 E=C\cup([T+1,Q]\setminus G),\qquad F=D\cup G.
\]
Then
\begin{equation}\label{eq1}
[T+1,Q]\subseteq hE,
\end{equation}
\begin{equation}\label{eq2}
[T+1,Q-1]\subseteq hF,
\end{equation}

\begin{equation}\label{eq3}
Q\notin hF,
\end{equation}
and
\begin{equation}\label{eq4}
Q\in h(F\cup\{c\})
\end{equation}
for every $c\in C$.
\end{lemma}

\begin{proof}
Since $[U,T]\subseteq D$, $0\in D$ and $C\sqcup D=[0,T]$, it follows that
\[
 C\subseteq[1,U-1].
\]
Noting that $1,2\in C$, we have $U\ge3$.
Moreover, since $c\le T$, $L=2T$ and $M=(2h-3)T+1$, it follows that
\[
 x_c-2M=(2h-2)M+1-(h-2)L-c>0.
\]
Thus
\[
 2M<x_c<Q.
\]

The hypothesis gives
\[
 [T+1,T+h-1]\subseteq hC\subseteq hE.
\]
By $[T+1,L-1]\subseteq E$ and $1,2\in E$, we have 
\[
 [T+h,L+h-2]\subseteq hE.
\]
Applying Lemma~\ref{lem1} to $S=E$, $p=2$, $u=T+1$ and
\[v=L-1=2T-1=2(T+1)-2-1,
\]
we obtain
\[
 [T+2h-1,h(L-1)]\subseteq hE.
\]
Noting that $T\ge T_0$, these three intervals overlap or are adjacent.  Consequently,
\[
 [T+1,h(L-1)]\subseteq hE.
\]

Next, $[L+1,M-1]\subseteq E$, and $T+1\in E$. Applying Lemma~\ref{lem1},
now with $p=T+1$, $u=L+1$ and $v=M-1$, we obtain
\[
 [L+1+(h-1)(T+1),h(M-1)]\subseteq hE,
\]
since
\[
 M-1=(2h-3)T\ge3T=2(L+1)-(T+1)-1.
\]
Moreover,
\[
 L+1+(h-1)(T+1)=(h+1)T+h\le h(L-1)+1,
\]
where the last inequality follows from $T\ge T_0$.  Hence
\[
 [T+1,h(M-1)]\subseteq hE.
\]

Noting that $C\subseteq[1,U-1]$, we have
\[
 x_c\ge Q-(h-2)L-(U-1)
\]
for each $c\in C$.
Therefore, if
\[
 V=Q-(h-2)L-U,
\]
then
\[
 [2M+1,V]\subseteq E.
\]
Also
\[
 V-4M=2(h-2)(M-T)+1-U>0,
\]
so Lemma~\ref{lem1} applied to $S=E$, $p=1$, $u=2M+1$,
$v=V$ and $m=2$ yields
\[
 [2M+2,2V]\subseteq2E.
\]
It follows that
\[
 [2M+h,2V+h-2]\subseteq hE.
\]
Since $(h-2)M\ge2h-1$, it follows that
\[
2M+h\leq h(M-1)+1.
\]  Noting that
$2U\le T+1$, we have
\begin{align*}
 2V+h-2-Q
 &=Q-2(h-2)L-2U+h-2\\
 &\ge(4h^2-10h+7)T+3h-2>0.
\end{align*}  Thus,
\[
 [T+1,Q]\subseteq hE,
\]
which proves \eqref{eq1}.

By $0\in D$, $[U,T]\subseteq D$ and $T\ge2U-1$, Lemma~\ref{lem1}
with $S=D$ and $p=0$, we have
\[
 [U,mT]\subseteq mD
\]
 for every $m\ge2$.
For $0\le r\le h-2$, take $r$ copies of $L$ and $h-r$ summands from
$D$, we obtain
\[
 [rL+U,rL+(h-r)T]
   =[2rT+U,(h+r)T]\subseteq hF.
\]
For $0\le r\le h-3$, we have
\[
 (h+r)T+1-\bigl(2(r+1)T+U\bigr)
 =(h-r-2)T+1-U\ge0.
\]
It follows that 
\[
\bigcup_{0\le r\le h-2}[2rT+U,(h+r)T]=[U,(2h-2)T]\subseteq hF.
\]
  Since $M-1=(2h-3)T$ and $T+1\geq 2U$, it follows that
\[
 [T+1,M-1]\subseteq hF.
\]

On the other hand, $0\in F$ and $[M,2M]\subseteq F$.  For each
$1\le s\le h$, we have
\[
 s[M,2M]+(h-s)\cdot0=[sM,2sM]\subseteq hF.
\]
It follows that 
\[
\bigcup_{1\le s\le h}[sM,2sM]=[M,2hM]=[M,Q-1]\subseteq hF,
\]
which proves \eqref{eq2}.

For every $x_c$, we have
\[
 x_c-\frac Q2= \frac Q2-(h-2)L-c
 \ge hM-2(h-2)T-T>0.
\]
Thus a representation of $Q$ by $h$ elements of $F$ could contain at
most one element of the form $x_c$. If it contains no $x_c$, then every summand is at most $2M$, and so
 $Q\leq 2hM=Q-1$, a contradiction. Therefore,
one summand is $x_c$. It follows that
\[
 Q-x_c=(h-2)L+c\in (h-1)F.
\]
Since $C\subseteq[1,U-1]$, it follows that
\[
 (h-2)L+c\le2(h-2)T+U-1\le(2h-3)T=M-1.
\]
By the definition of $F$ and $x_d>2M$ for every $d\in C$, we have
\[
 F\cap[0,M-1]=D\cup\{L\}.
\]
Consequently,
\[
 Q-x_c=(h-2)L+c\in (h-1)(D\cup\{L\}).
\]
Let $r$ be the number of copies of $L$ among these $h-1$ summands.  If
$r\le h-3$, then
\[
 Q-x_c=(h-2)L+c\leq rL+(h-1-r)T=(h+r-1)T\le(2h-4)T=(h-2)L,
\]
which is a contradiction because $c\ge1$.  If
$r=h-2$, then $c\in D\cup\{L\}$.
Noting that $L=2T$, we have $c\in D$, a contradiction with $C\cap D=\varnothing$.
If $r=h-1$, the sum equals $(h-2)L+c=(h-1)L$, a contradiction with $c<T$. So $Q\notin hF$,
which proves \eqref{eq3}.

For every $c\in C$, by $x_c,L\in F$, we have
\[
 Q=x_c+\underbrace{L+\cdots+L}_{h-2\text{ copies}}+c\in h(F\cup\{c\}),
\]
which proves \eqref{eq4}.

This completes the proof of Lemma \ref{lem2}.
\end{proof}

\begin{lemma}\label{lem3}
Let $T$ and $U$ be two positive integers with $T\ge T_0$ and $T=2U-2$, and let $C$ and $D$ be two sets such that
\[
 0\in D,\qquad 1\in C,\qquad [d,d+2h]\subseteq D,\qquad C\sqcup D=[0,T],\qquad [U,T]\subseteq C.
\]
Put
\[
 L=2T,\qquad M=(2h-3)T+1,\qquad W=2hM+1.
\]
Fix $c\in C$ and put
\[
 y_c=W-(h-2)L-c.
\]
Put
\[
 H=\{L\}\cup[M,2M]\cup\{y_c\}\cup[W-h+2,W].
\]
We define
\[
 E=C\cup H,\qquad F=D\cup([T+1,W]\setminus H).
\]
Then
\begin{equation}\label{eq5}
[T+1,W]\subseteq hE\cap hF,
\end{equation}
and
\begin{equation}\label{eq6}
W\notin h(E\setminus\{c\}).
\end{equation}
\end{lemma}

\begin{proof}
Since $0\in D$, $c\in C$ and $C\sqcup D=[0,T]$, it follows that
\[
 1\le c\le T.
\]
Noting that
\[
 (h-2)L+T=(2h-3)T=M-1,
\]
we have
\[
 y_c\ge W-(M-1)=(2h-1)M+2>2M.
\]
Moreover,
\[
 W-h+2-y_c=(h-2)L+c-h+2>0.
\]
Consequently,
\[
 2M<y_c<W-h+2,
\]
so the four pieces occurring in the definition of $H$ are pairwise disjoint.

By the definition of $y_c$, we have
\[
 W=y_c+\underbrace{L+\cdots+L}_{h-2\text{ copies}}+c\in hE.
\]
We first prove \eqref{eq6}. Noting that $0\in D$ and $C\sqcup D=[0,T]$, we have $0\notin E$. It follows that $e>0$ for every $e\in E$. Write
\[W=w_1+w_2+\ldots+w_h,~~~ w_i\in E,~~~i=1,2,\ldots,h.
\]
If there exists an integer $i_0\in\{1,2,\ldots,h\}$ such that
\[
 w_{i_0}\in[W-h+2,W],
\]
then 
\[
 W\ge w_{i_0}+h-1\ge W+1,
\]
a contradiction. Thus, $w_i\not\in [W-h+2,W]$ for $i=1,2,\ldots,h$.

Since
\[
 y_c-\frac W2\ge (h-1)M+\frac32>0,
\]
it follows that
\[
\sum\limits_{i=1}^{h}|\{w_i\}\cap\{y_c\}|\leq 1.
\] 
If \[
\sum\limits_{i=1}^{h}|\{w_i\}\cap\{y_c\}|=0,
\] 
then 
\[
 W\le 2hM=W-1,
\]
a contradiction. Therefore,
\[
\sum\limits_{i=1}^{h}|\{w_i\}\cap\{y_c\}|=1.
\] 
It follows that $W-y_c\in (h-1)E$ and
\[
 W-y_c=(h-2)L+c\le(h-2)L+T=M-1.
\]
By the definition of $E$, we have
\[
 E\cap[0,M-1]=C\cup\{L\},
\]
and so \[
W-y_c\in (h-1)(C\cup\{L\}).
\] 
Let $r$ be the number of copies of $L$ among the remaining $h-1$ summands. If $r\le h-3$, then
\[
W-y_c\leq rL+(h-1-r)T=(h+r-1)T\le(2h-4)T=(h-2)L,
\]
which is impossible because $c\ge1$. If $r=h-1$, then
\[
 (h-2)L+c=(h-1)L,
\]
which is impossible because $c\le T<L$. Thus $r=h-2$, so there exists an integer $i_1\in\{1,2,\ldots,h\}$ such that $w_{i_1}=c$, which proves \eqref{eq6}.

We next prove \eqref{eq5}. Noting that $1\in E$, $[U,T]\subseteq E$ and 
\[
 T=2U-2=2U-1-1,
\]
Lemma~\ref{lem1} applies with $S=E$, $p=1$, $u=U$, $v=T$ and $m=h-k$ for every $0\le k\le h-2$.
Then
\[
 I_k=[2kT+U+h-k-1,(h+k)T]\subseteq hE.
\]
Noting that $U=(T+2)/2$ and $T\ge T_0$, we have $U\ge h+2$, and hence
\[
 \min I_0=U+h-1\le T+1.
\]
For $0\le k\le h-3$, by $T=2U-2$ and $U\ge h+2$, we have
\[
 U+h-k-3\le(h-k-2)T.
\]
Moreover,
\[
 \max I_{h-2}=(2h-2)T\ge(2h-3)T=M-1.
\]
Consequently,
\[
 [T+1,M-1]\subseteq hE.
\]
Noting that $T,1,L\in E$, we have
\[
 M=\underbrace{L+\cdots+L}_{h-2\text{ copies}}+T+1\in hE.
\]
For $1\le j\le h-2$, noting that
\[
 T+1+j\in[2U,2T]=2[U,T],
\]
we have
\[
 M+j=(h-2)L+(T+1+j)\in hE.
\]
Finally, Lemma~\ref{lem1} applied to $S=E$, $p=1$, $u=M$ and $v=2M$ gives
\[
 [M+h-1,2hM]=[M+h-1,W-1]\subseteq hE.
\]
It follows from $W\in hE$ that
\[
 [T+1,W]\subseteq hE.
\]
Noting that $e\in F$ for every $e\in [T+1,W]\setminus H$ and $0\in F$, we have $e\in hF$ for every $e\in [T+1,W]\setminus H$. Thus it is enough to cover the four pieces of $H$.

Noting that $d\in D$, $d<T$ and $L=2T$, we have
\[
 T<L-d<L.
\]
Hence, $L-d\in[T+1,L-1]\subseteq F$, and so
\[
 L=(L-d)+d+(h-2)\cdot0\in hF.
\]

Next, put
\[
 J=[L+1,M-1]=[2T+1,M-1]\subseteq F.
\]
Suppose $h\ge4$. Since
\[
 4T+2\le M,
\]
it follows that
\[
 [M,2M-2]\subseteq2J=[4T+2,2M-2]\subseteq hF.
\]
Moreover,
\[
 6T+3\le2M-1\le2M\le3M-3,
\]
so
\[
 2M-1,2M\in3J\subseteq hF.
\]
Thus,
\[
 [M,2M]\subseteq hF.
\]

If $h=3$, then $M=3T+1$ and $J=[2T+1,3T]$. It follows that
\[
 [T+1,2T-1]+J=[3T+2,5T-1]\subseteq hF
\]
and
\[
 2J=[4T+2,6T]\subseteq hF.
\]
Thus, $[M+1,2M-2]\subseteq hF$. Since 
\[
 d,d+1,d+2\in[d,d+2h]\subseteq D,
\]
while $M-d,3T-d\in J$, it follows that
\[
 M=d+(M-d)+(h-2)\cdot0\in hF,
\]
and
\[
 2M-1=3T+(3T-d)+(d+1)\in hF,
\]
\[
 2M=3T+(3T-d)+(d+2)\in hF.
\]
Hence, $[M,2M]\subseteq hF$.

Since $d+2h=12h$ and $T\ge T_0$, it follows that
\[
 y_c-(2M+d+2h)\ge(2h-3)M+2-12h>0.
\]
Thus for every $b\in[d,d+2h]$, we have
\[
 2M<y_c-b<y_c<W-h+2.
\]
By the construction, $y_c-b\in F$, we have
\[
 y_c=(y_c-b)+b+(h-2)\cdot0\in hF.
\]
Let $z\in[W-h+2,W]$. Choose $b\in[d,d+2h]$ such that $z-b\ne y_c$. Since $d=10h>h-2$, it follows that
\[
 z-b\le W-d<W-h+2.
\]
On the other hand,
\[
 z-b\ge W-h+2-(d+2h)=2hM+3-13h>2M.
\]
Therefore,
\[
 2M<z-b<W-h+2,
\]
and, by the choice of $b$, the integer $z-b$ does not belong to any of the four pieces of $H$. Hence, $z-b\in F$ and
\[
 z=(z-b)+b+(h-2)\cdot0\in hF.
\]
It follows that
\[
 [T+1,W]\subseteq hF.
\]
Combining the two inclusions proves \eqref{eq5}. 

This completes the proof of Lemma \ref{lem3}.
\end{proof}

\begin{lemma}\label{lem4}
Let sets $E,F$ and integers $Q$ be as defined in Lemma \ref{lem2}. Put
\[
 R=2Q+2.
\]
We define
\[
X=E\cup [R+1,2R],~~~Y=F\cup [Q+1,R].
\]
Then
\begin{equation}\label{eq7}
[Q+1,2R]\subseteq hX\cap hY.
\end{equation}

\end{lemma}

\begin{proof}
Since $0\in F\subseteq Y$ and $[Q+1,R]\subseteq Y$, it follows that
\[
 [Q+1,R]\subseteq hY.
\]
Moreover,
\[
 2[Q+1,R]=[2Q+2,2R]=[R,2R].
\]
Thus, 
\[
 [R+1,2R]\subseteq hY.
\]
Consequently,
\[
 [Q+1,2R]\subseteq hY.
\]

In the proof of Lemma~\ref{lem2}, we put
\[
 V=Q-(h-2)L-U.
\]
Then
\[
 [2M+1,V]\subseteq E.
\]
Since $L=2T$, $M=(2h-3)T+1$ and $T\ge2U-1$, it follows that
\[
 V-\bigl((2h-1)M+1\bigr)
 =M-(h-2)L-U=T+1-U\ge0.
\]
Hence,
\[
 [2M+1,(2h-1)M+1]\subseteq E\subseteq X.
\]
Noting that $1\in E\subseteq X$, we apply Lemma~\ref{lem1} to $S=X$, $p=1$,
$u=2M+1$ and $v=(2h-1)M+1$. It follows from
\[
 (2h-1)M+1-\bigl(2(2M+1)-1-1\bigr)
 =(2h-5)M+1>0,
\]
that
\[
 [2M+h,\,h((2h-1)M+1)]\subseteq hX.
\]
Since $Q=2hM+1$ and $R=2Q+2=4hM+4$, it follows that
\[
 Q+1-(2M+h)=2(h-1)M-h+2>0
\]
and
\[
 h((2h-1)M+1)-(R+h-1)=h(2h-5)M-3\ge0.
\]
Thus,
\[
 [Q+1,R+h-1]\subseteq hX.
\]

On the other hand, $[R+1,2R]\subseteq X$ and $1\in X$. Applying Lemma~\ref{lem1}
with $S=X$, $p=1$, $u=R+1$ and $v=2R$. It follows from
\[
 2R=2(R+1)-1-1
\]
that
\[
 [R+h,2hR]\subseteq hX.
\]
Consequently,
\[
 [Q+1,2R]\subseteq hX.
\]
Combining the two inclusions proves \eqref{eq7}. 

This completes the proof of Lemma \ref{lem4}.
\end{proof}

\begin{lemma}\label{lem5}
Let sets $E,F$ and integers $W$ be as defined in Lemma \ref{lem3}. Put 
\[
 R=2W+2.
\]
We define
\[
X=E\cup [W+1,R],~~~Y=F\cup [R+1,2R+1].
\]
Then
\begin{equation}\label{eq8}
[W+1,2R+1]\subseteq hX\cap hY
\end{equation}
and
\begin{equation}\label{eq9}
[2R+2,2R+h]\subseteq hX.
\end{equation}
\end{lemma}

\begin{proof}
Noting that $[W-h+2,W]\subseteq E\subseteq X$ and $1\in E\subseteq X$,
for every $1\le j\le h-1$, we have
\[
 W+j=(W-h+1+j)+(h-1)\cdot1\in hX.
\]
Hence,
\[
 [W+1,W+h-1]\subseteq hX.
\]
Applying Lemma~\ref{lem1} to $S=X$, $p=1$,
$u=W+1$ and $v=R$. By $[W+1,R]\subseteq X$ and \[
 R=2W+2>2(W+1)-1-1=2W,
\]
we have
\[
 [W+h,hR]\subseteq hX.
\]
Moreover, by $h\ge3$ and $R>T\ge T_0$, we have
\[
 hR-(2R+h)=(h-2)R-h\ge0.
\]
 Therefore,
\[
 [W+1,2R+h]\subseteq hX.
\]
In particular,
\[
 [W+1,2R+1]\subseteq hX
\]
and
\[
 [2R+2,2R+h]\subseteq hX.
\]

By Lemma~\ref{lem3}, we have
\[
 2M<y_c<W-h+2.
\]
Thus, by the definition of $F$, we have
\[
 K=[2M+1,y_c-1]\subseteq F\subseteq Y.
\]
Since $c\le T$ and $(h-2)L+T=M-1$, it follows that
\[
 y_c-1\ge W-(M-1)-1=(2h-1)M+1.
\]
In particular,
\[
 y_c-1\ge4M+1=2(2M+1)-1.
\]
Since $0\in F\subseteq Y$, Lemma~\ref{lem1} applied to $S=Y$, $p=0$,
$u=2M+1$ and $v=y_c-1$ gives
\[
 [2M+1,h(y_c-1)]\subseteq hY.
\]
Furthermore,
\[
 h(y_c-1)-R
 \ge h(2h-5)M+h-4>0,
\]
where $R=2W+2=4hM+4$. Since $W+1>2M+1$, it follows that
\[
 [W+1,R]\subseteq hY.
\]

On the other hand, noting that $[R+1,2R+1]\subseteq Y$ and $0\in Y$, we have
\[
 [R+1,2R+1]\subseteq hY.
\]
Consequently,
\[
 [W+1,2R+1]\subseteq hY.
\]
Hence, \eqref{eq8} and \eqref{eq9} are true.

This completes the proof of Lemma \ref{lem5}.
\end{proof}
\section{Proof of Theorem~\ref{thm1}}
The case $h=2$ follows from the theorem of Erd\H{o}s and Nathanson on infinitely oscillating bases and nonbases \cite{ErdosNathanson1976}. We therefore assume that $h\ge3$ and retain the notation
\[
 d=10h,\qquad T_0=100h^2.
\]
Choose an integer $W_0\ge T_0$ such that
\[
 W_0>d+2h.
\]
Put
\[
 R_0=2W_0+2,\qquad T_1=2R_0+1,\qquad U_1=R_0+1.
\]
Choose a partition $A_1\sqcup B_1=[0,T_1]$ such that
\[
 1,2,3,4\in A_1,\qquad 0\in B_1,\qquad [d,d+2h]\subseteq B_1,\qquad [W_0+1,R_0]\subseteq A_1,\qquad [R_0+1,T_1]\subseteq B_1.
\]
Noting that $U_1=R_0+1$, we have
\[
 [U_1,T_1]\subseteq B_1,\qquad T_1=2U_1-1.
\]
Moreover, $1\in A_1$ and $[W_0+1,R_0]\subseteq A_1$. Applying Lemma~\ref{lem1} to $S=A_1$, $p=1$, $u=W_0+1$ and $v=R_0$, we have
\[
 [W_0+h,hR_0]\subseteq hA_1.
\]
Noting that $W_0\ge T_0$, we have
\[
 W_0+h\le T_1+1
\]
and
\[
 T_1+h-1=2R_0+h\le hR_0.
\]
Consequently,
\[
 [T_1+1,T_1+h-1]\subseteq hA_1.
\]
Thus $A_1$ and $B_1$ satisfy the hypotheses of Lemma~\ref{lem2} with $T=T_1$ and $U=U_1$.

We now construct the two sets recursively. Fix a sequence $\sigma(1),\sigma(2),\ldots$ of positive integers such that every positive integer occurs infinitely often and
\[
 \sigma(k)\le k
\]
for every $k\ge1$. Suppose that, for some $k\ge1$, the partition has been determined on $[0,T_k]$ as
\[
 A_k\sqcup B_k=[0,T_k],
\]
where
\[
 [U_k,T_k]\subseteq B_k,\qquad T_k=2U_k-1,
\]
and
\[
 [T_k+1,T_k+h-1]\subseteq hA_k.
\]

Applying Lemma~\ref{lem2} with $C=A_k$ and $D=B_k$, we obtain an integer $Q_k$ and a partition $E_k\sqcup F_k=[0,Q_k]$ such that
\[
 [T_k+1,Q_k]\subseteq hE_k,
\]
\[
 [T_k+1,Q_k-1]\subseteq hF_k
\]
and
\[
 Q_k\notin hF_k.
\]
Put
\[
 R_k=2Q_k+2,\qquad S_k=2R_k.
\]
Applying Lemma~\ref{lem4} to $E_k$ and $F_k$, we obtain an integer $S_k$ and a partition $[0,S_k]=C_k\sqcup D_k$ such that
\[
 [Q_k+1,S_k]\subseteq hC_k\cap hD_k
\]
and
\[
 [R_k+1,S_k]\subseteq C_k,\qquad S_k=2(R_k+1)-2.
\]
Also, we have
\[
 0\in D_k,\qquad 1\in C_k,\qquad [d,d+2h]\subseteq D_k.
\]
Thus $C_k$ and $D_k$ satisfy the hypotheses of Lemma~\ref{lem3} with $T=S_k$ and $U=R_k+1$.

Since each interval $[R_j+1,S_j]\neq\emptyset$, before the $k$-th stage at least $k$ elements of the eventual set $A$ have already been determined. Moreover, every element assigned later is larger than $S_k$. Hence the first $k$ elements of the eventual set $A$ are already fixed. Let $c_k$ be the $\sigma(k)$-th smallest element of $C_k$.

Applying Lemma~\ref{lem3} with $C=C_k$, $D=D_k$ and $c=c_k$, we obtain an integer $W_k$ and a partition $\widetilde E_k\sqcup\widetilde F_k=[0,W_k]$ such that
\[
 [S_k+1,W_k]\subseteq h\widetilde E_k\cap h\widetilde F_k
\]
and
\[
 W_k\notin h(\widetilde E_k\setminus\{c_k\}).
\]
Put
\[
 P_k=2W_k+2,
\]
and apply Lemma~\ref{lem5}. Then we obtain a partition $[0,T_{k+1}]=A_{k+1}\sqcup B_{k+1}$, where
\[
 T_{k+1}=2P_k+1,\qquad U_{k+1}=P_k+1.
\]
By the definition of $B_{k+1}$,
\[
 [U_{k+1},T_{k+1}]\subseteq B_{k+1},\qquad T_{k+1}=2U_{k+1}-1.
\]
Also, by \eqref{eq9}, we have
\[
 [T_{k+1}+1,T_{k+1}+h-1]\subseteq hA_{k+1}.
\]
Thus $A_{k+1}$ and $B_{k+1}$ again satisfy the hypotheses of Lemma~\ref{lem2}. 

We next prove by induction that, for every integer $k\ge1$, we have
\begin{equation}\label{eq10}
[T_1+1,T_k]\subseteq hA_k,
\qquad
[T_1+1,T_k]\setminus\{Q_1,\ldots,Q_{k-1}\}\subseteq hB_k,
\end{equation}
and
\begin{equation}\label{eq11}
Q_j\notin hB_k\qquad(1\le j\le k-1).
\end{equation}
For $k=1$, \eqref{eq10} and \eqref{eq11} are true. Assume that they hold for some $k\ge1$. Since every later block preserves all previously assigned elements, the conclusions of Lemmas~\ref{lem2}--\ref{lem5} remain valid in $A_{k+1}$ and $B_{k+1}$. By Lemma~\ref{lem2}, we have
\[
 [T_k+1,Q_k]\subseteq hA_{k+1}
\]
and
\[
 [T_k+1,Q_k-1]\subseteq hB_{k+1}.
\]
By Lemma~\ref{lem4}, we have
\[
 [Q_k+1,S_k]\subseteq hA_{k+1}\cap hB_{k+1}.
\]
By Lemma~\ref{lem3}, we have
\[
 [S_k+1,W_k]\subseteq hA_{k+1}\cap hB_{k+1}.
\]
Finally, Lemma~\ref{lem5} gives
\[
 [W_k+1,T_{k+1}]\subseteq hA_{k+1}\cap hB_{k+1}.
\]
Therefore,
\[
 \begin{split}
 [T_k+1,T_{k+1}]
 ={}&[T_k+1,Q_k]\cup[Q_k+1,S_k]\\
 &\cup[S_k+1,W_k]\cup[W_k+1,T_{k+1}]
 \subseteq hA_{k+1},
 \end{split}
\]
and
\[
 \begin{split}
 [T_k+1,T_{k+1}]\setminus\{Q_k\}
 ={}&[T_k+1,Q_k-1]\cup[Q_k+1,S_k]\\
 &\cup[S_k+1,W_k]\cup[W_k+1,T_{k+1}]
 \subseteq hB_{k+1}.
 \end{split}
\]
Combining these inclusions with the induction hypothesis proves \eqref{eq10} for $k+1$.

For every $1\leq j\leq k$, by Lemma~\ref{lem2}, we have $Q_j\notin hF_j$. Suppose that $Q_j\in hB_{k+1}$. Then there exist $b_1,\ldots,b_h\in B_{k+1}$ such that
\[
 Q_j=b_1+\cdots+b_h.
\]
Since $b_1,\ldots,b_h$ are nonnegative, it follows that
\[
 0\le b_i\le Q_j\qquad(1\le i\le h).
\]
Consequently,
\[
 b_i\in B_{k+1}\cap[0,Q_j]=F_j\qquad(1\le i\le h),
\]
and hence
\[
 Q_j=b_1+\cdots+b_h\in hF_j,
\]
a contradiction. Therefore,
\[
 Q_j\notin hB_{k+1}\qquad(1\le j\le k).
\]
This proves \eqref{eq11} for $k+1$. Thus \eqref{eq10} and \eqref{eq11} hold for every positive integer $k$.

Noting that $T_{k+1}>T_k$ for every $k$, we have $T_k\to\infty$. Put
\[
 A=\bigcup_{k\ge1}A_k,\qquad B=\bigcup_{k\ge1}B_k.
\]
By the definition of $A_k$ and $B_k$, we have
\[
 \Nzero=A\sqcup B.
\]
It follows from \eqref{eq10} that
\[
 [T_1+1,\infty)\subseteq hA.
\]
Thus $A$ is an asymptotic basis of order $h$.

By \eqref{eq11} and the argument above, we have
\[
 Q_k\notin hB
\]
for every positive integer $k$. By the construction, $Q_{k-1}<T_k<Q_k$ for $k\ge2$, and hence $Q_1<Q_2<\cdots$. So the set $hB$ misses infinitely many positive integers. Hence, $B$ is an asymptotic nonbasis of order $h$.

Fix $a\in A$. Since $T_k\to\infty$, there exists an integer $k_0$ such that
\[
 a\in A_k
\]
for every $k\ge k_0$. By \eqref{eq4}, we have
\[
 Q_k\in h(B\cup\{a\})
\]
for every $k\ge k_0$. On the other hand, by \eqref{eq10}, we have
\[
 [T_1+1,\infty)\setminus\{Q_1,Q_2,\ldots\}\subseteq hB.
\]
Consequently, every sufficiently large integer belongs to $h(B\cup\{a\})$. Hence, $B$ is a maximal asymptotic nonbasis of order $h$.

Write
\[
 A=\{a_1<a_2<\cdots\}.
\]
Fix $a_i\in A$. Since every positive integer occurs infinitely often in the sequence $\sigma$, there are infinitely many integers $k$ such that
\[
 \sigma(k)=i.
\]
For each such $k$, noting that $\sigma(k)\le k$, by the choice of $c_k$, we have
\[
 c_k=a_i.
\]
By Lemma~\ref{lem3}, we have
\[
 W_k\notin h(\widetilde E_k\setminus\{a_i\}).
\]
Suppose that $W_k\in h(A\setminus\{a_i\})$. Then there exist $u_1,\ldots,u_h\in A\setminus\{a_i\}$ such that
\[
 W_k=u_1+\cdots+u_h.
\]
Since $u_1,\ldots,u_h$ are nonnegative, we have
\[
 0\le u_i\le W_k\qquad(1\le i\le h).
\]
Thus,
\[
 u_i\in (A\setminus\{a_i\})\cap[0,W_k]
 =\widetilde E_k\setminus\{a_i\}
 \qquad(1\le i\le h),
\]
and therefore
\[
 W_k=u_1+\cdots+u_h\in h(\widetilde E_k\setminus\{a_i\}),
\]
a contradiction. Hence,
\[
 W_k\notin h(A\setminus\{a_i\}).
\]
Moreover, by the construction, we have
\[
 W_k<T_{k+1}<W_{k+1}
\]
for every $k\ge1$. Hence the infinitely many indices satisfying $\sigma(k)=i$ yield infinitely many distinct integers missing from $h(A\setminus\{a_i\})$. Consequently, $A\setminus\{a_i\}$ is not an asymptotic basis of order $h$. Hence, $A$ is minimal.

Thus $A$ is a minimal asymptotic basis of order $h$ and $B$ is a maximal asymptotic nonbasis of order $h$. 

This completes the proof of Theorem \ref{thm1}.

\end{document}